\documentclass[12pt,twoside]{amsart}
\usepackage{amsmath}
\usepackage{amsthm}
\usepackage{amsfonts}
\usepackage{amssymb}
\usepackage{latexsym}
\usepackage{mathrsfs}
\usepackage{amsmath}
\usepackage{amsthm}
\usepackage{amsfonts}
\usepackage{amssymb}
\usepackage{latexsym}
\usepackage{geometry}
\usepackage{dsfont}
\usepackage[dvips]{graphicx}
\usepackage{color}
\usepackage[all]{xy}

\date{}
\allowdisplaybreaks[4] \footskip=15pt
\renewcommand{\uppercasenonmath}[1]{}

\usepackage{graphicx,amssymb}
\usepackage[all]{xy}
\usepackage{amsmath}

\allowdisplaybreaks
\usepackage{amsthm}
\usepackage{color}

\theoremstyle{plain}
\newtheorem{theorem}{Theorem}[section]

\newtheorem{lemma}[theorem]{Lemma}
\newtheorem{corollary}[theorem]{Corollary}

\newtheorem*{open question}{Open Question}

\theoremstyle{definition}

\theoremstyle{remark}
\def\bc{\begin{center}}
\def\ec{\end{center}}

\def\ker{{\rm ker}}
\def\and{{\rm and}}

\def\Spec{{\rm Spec}}

\def\D{\textbf{D}}

\begin{document}
\begin{center}
{\large  \bf The Telescope Conjecture for von Neumann regular rings}

\vspace{0.5cm}   Xiaolei Zhang

{\footnotesize a.\ Department of Basic Courses, Chengdu Aeronautic Polytechnic, Chengdu 610100, China\\

E-mail: zxlrghj@163.com\\}
\end{center}

\bigskip
\centerline { \bf  Abstract}
\bigskip
\leftskip10truemm \rightskip10truemm \noindent

In this note, we show that any epimorphism originating at a von Neumann regular ring (not necessary commutative) is a universal localization. As an application, we prove that the Telescope Conjecture holds for the unbounded derived categories of von Neumann regular rings (not necessary commutative).
\vbox to 0.3cm{}\\
{\it Key Words:} derived category; ring epimorphism;
von Neumann regular ring; the Telescope Conjecture.\\
{\it 2010 Mathematics Subject Classification:} 16E50,18G80.

\leftskip0truemm \rightskip0truemm
\bigskip

In this note, $R$ is a ring with unit and $\D(R)$ denotes the unbounded derived category of $R$. The Telescope Conjecture in triangulated category began from Bousfield \cite{Bou79} and Ravenel\cite{Rav84} for the stable homotopy category in algebraic topology and it has been formulated in a more general setting of compactly generated triangulated category as follows (see \cite{SS08}).

\textbf{Telescope Conjecture for Triangulated Categories} \ \ {\sl
Every smashing localizing subcategory of a compactly generated triangulated category is generated by compact objects.}\\
In 1992,  Neeman \cite{Nee92} showed that the Telescope Conjecture holds for  unbounded derived category of  commutative noetherian rings, essentially by classifying all the localizing subcategories by prime ideals.  Since then, the conjecture has drawn lots of attention of many algebraists. In 2010, Krause et al. \cite{Kra10} showed the Telescope Conjecture also holds for derived category of all hereditary rings via Ext-orthogonal pairs. However, Keller \cite{Kel94} gave an example of a valuation domain where the conjecture does not hold in 1994. In order to get a more accurate conclusion,  Bazzoni and \v{S}\v{t}ov\'{i}\v{c}ek gave a simple discrimination for the commutative rings of weak global dimension at most one in \cite{Baz17} recently. In that paper,  for a  commutative ring $R$ of weak global dimension at most one, the Telescope Conjecture holds for $\D(R)$ if and only if there is no $p\in \Spec(R)$ such that $pR_p$ is a non-zero idempotent ideal in $R_p$. Consequently, the Telescope Conjecture for all commutative von Neumann regular rings. In this note, we prove the Telescope Conjecture  holds for all von Neumann regular rings (commutative or non-commutative) by classifying their epimorphisms.

\section{The main result}
Firstly, we propose two theorems in \cite{Baz17}, which provide a direction to prove the main result.

\begin{theorem}\label{1}{\rm ({\cite[Theorem 3.10]{Baz17}})} Let $R$ be an algebra of weak global dimension at most one over a commutative ring $k$, then the assignment:
$$f\mapsto \{X\in \D(R)|X\otimes^{\mathbb{L}}_R S=0\}$$
is a bijection between

(1) equivalence classes of homological epimorphisms $f: R\rightarrow S$ originating at $R$, and

(2) smashing localizing subcategories $\mathfrak{X}\subseteq \D(R)$.
\end{theorem}

\begin{theorem}\label{2}{\rm ({\cite[Theorem 4.5]{Baz17}})} Let $R$ be a right semihereditary ring , then the assignment:

$$f\mapsto \{X\in \D(R)|X\otimes^{\mathbb{L}}_R S=0\}$$
restricts to a bijection between

(1) equivalence classes of universal localizations $f: R\rightarrow S$ originating at $R$, and

(2) compactly generated localizing subcategories $\mathfrak{X}\subseteq \D(R)$.
\end{theorem}

According to these two theorems,  we have the following discrimination methods for  the Telescope Conjecture for unbounded derived categories of right semihereditary rings.

\begin{corollary}\label{5}
Let $R$ be a right semihereditary ring, the Telescope Conjecture holds for $\D(R)$ if and only if every homological epimorphism from $R$ is a universal localization.
\end{corollary}
We characterize the epimorphisms originating at a von Neumann regular ring (commutative or non-commutative).
\begin{lemma}\label{3}
Every epimorphism originating at a von Neumann regular ring is surjective.
\end{lemma}
\begin{proof}
Let $R$ be a von Neumann regular ring, $f:R\rightarrow S $ an epimorphism and $I=\ker(f)$. Consider the following natural commutative diagram, $$\xymatrix{
R \ar[r]^f \ar@{->>}[rd]_g &
S\\
&R/I\ar@{>->}[u]^h.}$$
Then $R/I$ is a von Neumann regular ring  and $h$ is an epimorphism. To show $h$ is an isomorphism, we just need to prove $h$ is surjective.
Let $b\in S$, then $b$ is in the dominion of $f$. That is, there exist matrices $X\in M_{1\times m}(S), Y\in M_{m\times n}(R)$ and $z\in M_{n\times l}(S)$ such that $s=X f(Y) Z$ by \cite[Theorem 5.11]{Baz17}. By adding zeros, we can assume $f(Y)$ is a $k=max(m,n)$-square matrix. Because $M_{k\times k}(R)$ is also a  von Neumann regular ring, $f(Y)=f(Y) T f(Y)$ for some $T\in M_{k\times k}(R)$. Consequently, $b=Xf(Y)Z=[Xf(Y)]T[f(Y)Z]\in R/I$ and $f$ is surjective.
\end{proof}
\begin{lemma}\label{31}
Let $R$ be a von Neumann regular ring, every epimorphism originating at $R$ is a universal localization.
\end{lemma}
\begin{proof}Let $f: R\twoheadrightarrow R/I$ be an epimorphism. It is well known that $I$ can be written as a direct union of a family  $\{I_{\lambda}|\lambda\in \Lambda\}$ of finitely generated ideals. We obtain the following commutative diagram

$$\xymatrix{
   0 \ar[r]^{} &I_{\lambda} \ar@{>->}[d]_{i_{\lambda}}\ar[r]^{\theta_{\lambda}} &R  \ar@{=}[d]_{}\ar[r]^{\sigma _{\lambda}} & R/I_{\lambda} \ar@{->>}[d]_{\pi_{\lambda}}\ar[r]^{} &  0\\
    0 \ar[r]^{} &I\ar[r]^{\theta} &R \ar[r]^{\sigma} & R/I\ar[r]^{} &  0\\}$$
Since $I_\lambda$ is finitely generated and $R$ is a von Neumann regular ring, $g_\lambda$ is a splitting monomorphism by \cite[Theorem 1.11]{Goo79}. Let $J_{\lambda}=R/I_{\lambda}$ such that $I_{\lambda}\oplus J_{\lambda}=R$, $\sigma _{\lambda}: J_{\lambda}\rightarrow R$ a homomorphism such that $s_\lambda \sigma _{\lambda}=Id_{J_{\lambda}}$ and $t _{\lambda}: R\rightarrow I_{\lambda}$ a homomorphism such that $t_\lambda \theta _{\lambda}=Id_{I_{\lambda}}$. Let $\Sigma =\{\sigma_{\lambda}|\lambda\in \Lambda\}$, we prove $f$ is equivalent to the universal localization $R\rightarrow R_{\Sigma}$.

Firstly, we show that $\sigma_{\lambda}\otimes_R R/I$ is an isomorphism for each $\lambda\in \Lambda$. Let $0\rightarrow J_{\lambda}\rightarrow R\rightarrow I_{\lambda}\rightarrow 0$ a split short exact sequence. We obtain that the short exact sequence $0\rightarrow J_{\lambda}\otimes_R R/I\rightarrow R\otimes_R R/I\rightarrow I_{\lambda}\otimes_R R/I\rightarrow 0$ is split. Since $R/I$ is flat and $I_{\lambda}\subseteq I$, we get $I_{\lambda}\otimes_R R/I=I_{\lambda}R/I=0$. Consequently, $\sigma_{\lambda}\otimes_R R/I$ is an isomorphism.

Secondly, let $g: R\rightarrow S$ be any ring homomorphism such that $\sigma_{\lambda}\otimes_RS$ is an $S$-isomorphism for any $\sigma_{\lambda}\in \Sigma$. We will show that $g$ factors uniquely through $\sigma$. Similar to the proof of the first part, we have $I_{\lambda}\otimes_R S=0$, thus $g\theta(I)\leq IS=I\otimes_R S= \lim\limits_{\rightarrow} I_{\lambda}\otimes_R S =0$. Consequently, there is a unique homomorphism $h:R/I\rightarrow S$ such that the following diagram commutes,

$$\xymatrix{
 & & &S&\\
    0 \ar[r]^{} &I\ar[r]^{\theta} &R \ar[r]^{\sigma} \ar[ru]^{g}& R/I\ar[r]^{} \ar@{-->}[u]^{h}&  0\\}$$.
\end{proof}
As an application, we get the following main result.

\begin{theorem}\label{4} Let $R$ be a  von Neumann regular ring $($commutative or non-commutative$)$, then the Telescope Conjecture holds for $\D(R)$.
\end{theorem}
\begin{proof}
By Lemma \ref{31}, every homological epimorphism is a universal localization. Subsequently, the Telescope Conjecture holds for $\D(R)$ by Corollary \ref{5}.
\end{proof}

\end{document}